\documentclass[reqno,11pt]{amsart}
 
\usepackage{amsmath}
\allowdisplaybreaks[1]
\usepackage{amsthm}  
\usepackage{verbatim}
\usepackage{enumerate} 
\usepackage{mathtools}  
\usepackage{amssymb}
\usepackage{mathrsfs}
\usepackage[top=1.5in, bottom=1.5in, left=0.7in, right=0.7in]{geometry}  
\usepackage[colorlinks]{hyperref}
\hypersetup{
	colorlinks,
	citecolor=blue,
	linkcolor=red
}   
\numberwithin{equation}{section}
\usepackage{xypic}
\usepackage{bm}
\usepackage{tikz}
\usepackage{float}

\newtheorem{theorem}{\textbf{Theorem}}[section]
\newtheorem{theorem*}{\textbf{Theorem}}

\newtheorem{proposition}[theorem]{\textbf{Proposition}}

\newtheorem{claim}[theorem*]{\textbf{Claim}}

\newtheorem{corollary}[theorem]{\textbf{Corollary}}
\newtheorem{remark}[theorem]{\textbf{Remark}}

\newtheorem{definition/proposition}[theorem]{\textbf{Definition/Proposition}}
\newtheorem{innercustomgeneric}{\customgenericname}
\providecommand{\customgenericname}{}
\newcommand{\newcustomtheorem}[2]{%
	\newenvironment{#1}[1]
	{%
		\renewcommand\customgenericname{#2}%
		\renewcommand\theinnercustomgeneric{##1}%
		\innercustomgeneric
	}
	{\endinnercustomgeneric}
}
\newcustomtheorem{customthm}{Theorem}
\newcustomtheorem{customprop}{Proposition}
\newcustomtheorem{customcoro}{Corollary}

\def\N{{\mathbb N}}
\def\R{\mathbb{R}}
\def\Z{{\mathbb Z}}

\def\C{{\mathbb C}}

\def\Q{{\mathbb Q}}

\def\std{\rm{std}}

\DeclareMathOperator{\rank}{rank}

\newcommand{\Addresses}{{% additional braces for segregating \footnotesize
		\bigskip
		\footnotesize
		
	     Zhengyi Zhou, \par\nopagebreak
	    \textsc{Morningside Center of Mathematics and Institute of Mathematics, AMSS, CAS, China}\par\nopagebreak
		\textit{E-mail address}: \href{mailto:zhyzhou@amss.ac.cn}{zhyzhou@amss.ac.cn}

}}

\title{Infinite not contact isotopic embeddings in $(S^{2n-1},\xi_{\std})$ for $n\ge 4$}
\author{Zhengyi Zhou}

\begin{document}
	\maketitle
\begin{abstract}
For $n\ge 4$, we show that there are infinitely many formally contact isotopic embeddings of $(ST^*S^{n-1},\xi_{\std})$ to  $(S^{2n-1},\xi_{\std})$ that are not contact isotopic. This resolves a conjecture of Casals and Etnyre \cite{Nonsimple} except for the $n=3$ case. The argument does not appeal to the surgery formula of critical handle attachment for Floer theory/SFT.

\end{abstract}
%\tableofcontents
\section{Introduction}

Isocontact embeddings in dimension $3$, a.k.a.\ transverse links, are a fundamental object in the study of contact $3$-folds \cite{contact}. Their higher dimensional analogues are less studied until recently. A recent breakthrough in this direction is due to Casals and Etnyre \cite{Nonsimple}, who proved that isocontact embeddings of codimension $2$ submanifolds is not simple, i.e.\ being contact isotopic is finer than the underlying topological information (being formally contact isotopic). Such rigidity result emphasizes the fundamental difference between isocontact embeddings of codimension $2$, where non-topological obstructions appear, and isocontact embeddings of codimension at least $4$, which are governed by $h$-principle \cite{h-principle}, hence completely determined by their formal topological data. More precisely, Casals and Etnyre \cite[Theorem 1.1]{Nonsimple} showed that there are two contact embeddings of $(ST^*S^{n-1},\xi_{\std})$, i.e.\ the unit sphere bundle of the cotangent bundle of $S^{n-1}$ equipped with the canonical Liouville structure,  into  $(S^{2n-1},\xi_{\std})$ for $n\ge 3$ that are formally contact isotopic but not contact isotopic.  Then they conjectured \cite[Conjecture 1.5]{Nonsimple} that there are infinitely many formally contact isotopic embeddings of the standard $ST^*S^{n-1}$ into $(S^{2n-1},\xi_{\std})$ that are not contact isotopic to each other. In this note, we give a proof of their conjecture for $n\ge 4$.

\begin{theorem}\label{thm:main}
For each $n\ge 4$, there exists infinitely many contact embeddings of the standard $ST^*S^{n-1}$ into $(S^{2n-1},\xi_{\std})$ that are formally isotopic but not contact isotopic.
\end{theorem}

Let us briefly recall the geometric constructions in \cite{Nonsimple}. Starting from a Legendrian sphere $\Lambda$ in $(S^{2n-1},\xi_{\std})$, using the Weinstein neighborhood theorem, we get a contact push-off, i.e.\ a contact embedding of the standard $ST^*S^{n-1}$ into $(S^{2n-1},\xi_{\std})$. When $\Lambda,\Lambda'$ are formally Legendrian isotopic, the contact embeddings are also formally isotopic. Then Casals and Etnyre considered the branched double cover of $S^{2n-1}$ with branching locus $ST^*\Lambda$, where contact isotopic embeddings will induce contactormophic branched covers. The key fact in \cite{Nonsimple} is that the branched cover is precisely the contact manifold obtained by attaching a critical handle along the Legendrian sum $\Lambda \# \Lambda$. Therefore the problem of finding formally isotopotic but not contact isotopic emebeddings is reduced to finding formally isotopic Legendrians with different contact boundaries after the surgery along $\Lambda\#\Lambda$. Finally, Casals and Etnyre found a pair of such Legendrians $\Lambda,\Lambda'$ with $\Lambda'$ loose, such that the resulting contact manifolds are $ST^*S^n$ and $\partial (\mathrm{Flex}(T^*S^n))$, which are different by \cite{vanishing}.

\begin{remark}
There is another method to distinguish contact submanifolds via studying contact homology coupled with the intersection data with the holomorphic hypersurface given by the symplectization of the contact submanifold. Such invariants were introduced in \cite{cote2020homological} by C\^ot\'e and Fauteux-Chapleau, who used them to provide an alternative proof that some of different contact submanifolds built in \cite{Nonsimple} via contact push-off are indeed not contact isotopic, reproving \cite[Theorem 1.1]{Nonsimple} for $(S^{4n-1},\xi_{\std})$ and $n>1$.
\end{remark}

With the geometric constructions above, to prove \cite[Conjecture 1.5]{Nonsimple}, Casals and Etnyre  suggested to find infinitely many distinct but formally isotopic Legendrian spheres, then appeal to the surgery formula \cite{surgery} to show that they result in different contact boundaries after the surgery along $\Lambda \# \Lambda$. This is certainly plausible given the richness of formally isotopic Legendrian spheres on $(S^{2n-1},\xi_{\std})$. However, it is still a nontrivial task to compute their holomorphic curve invariants even using \cite{surgery}. Moreover, computing augmented invariants like symplectic cohomology as in \cite{surgery} is not sufficient to tell the differences of the contact boundaries. In fact, the latter is the main reason why we have to restrict to the case of $n\ge 4$ in this note. More precisely, we solve the $n\ge 4$ case by applying their geometric construction to the Legendrian spheres arising from the construction of exotic $T^*S^n$ by Eliashberg, Ganatra and Lazarev \cite{flexible}. More precisely, those exotic $T^*S^n$ are constructed from attaching a critical handle to formally isotopic Legendrian spheres that are not Legendrian isotopic. Those Legendrian spheres are boundaries of flexible Lagrangians of different diffeomorphism types in $\C^n$, whose existence is confirmed by $h$-principles in \cite{flexible}. The resulting contact manifolds are distinguished by their positive symplectic cohomology, which is a contact invariant for those asymptotically dynamically convex manifolds admitting Weinstein fillings \cite{ADC}. We note here that our proof does not rely on the surgery formula for critical handle attachments.

In the $n=3$ case, even though we do not find infinitely many not contact isotopic embeddings of $ST^*S^2$, we do have at least one more that is different from the two examples from \cite{Nonsimple}.

\begin{proposition}\label{prop:three}
There are three formally contact isotopic embeddings of $ST^*S^2$ into $(S^5,\xi_{\std})$ that are pairwisely not contact isotopic.
\end{proposition}

\subsection*{Acknowledgments}
The author would like to thank Roger Casals for his comments and interests, and Ruizhi Huang for helpful conversations.

\section{Proof of the $n\ge 4$ case}
The starting point is the flexible Lagrangian (with Legendrian boundary) $\hat{L}$ in the unit ball $\mathbb{D}^n\subset \mathbb{C}^n$ in the construction of the exotic $T^*S^n$ \cite[Theorem 4.7]{flexible}. Here we list the properties of them.
\begin{enumerate}
    \item $\hat{L}_k$ is a closed manifold $L_k$ minus a disk. The Legendrian boundary of $\partial \hat{L}_k$ is formally Legendrian isotopic to the standard unknot in $(S^{2n-1},\xi_{\std})$. 
    \item $\hat{L}_k$ is a flexible Lagrangian. In particular, when we attach a critical handle along $\partial \hat{L}_k$, we get an exotic $T^*S^n$, which is obtained from $T^*L_k$ and a flexible Weinstein cobordism.
    \item $\pi_1(L_k)=0$ and $H^2(L_k)=\Z^{2k}$. In view of \cite[Theorem 4.7]{flexible}, to have such $\widehat{L}_k$, in particular, to have $\partial \hat{L}_k$ is formally Legendrian isotopic to the standard unknot in $(S^{2n-1},\xi_{\std})$, we require $L_k$ to have the following properties.
    \begin{enumerate}
        \item $L_k$ is stably trivializable, which will imply that $TL_k\otimes \C$ is a trivial complex bundle.
        \item $\chi(L_k)=2$ when $n$ is even.
        \item $\chi_{\frac{1}{2}}(L_k):=\sum_{i=0}^m \rank H_i(L_k)=1\mod 2$ when $n=2m+1>3$.
    \end{enumerate}
    We can build such $L_k$ as the boundary of a regular neighborhood of an embedding of a CW complex $W_k$ in $\R^{n+1}$ with $2k$ $2$-cells and $2k$ $3$-cells with trivial attaching maps. Then $L_k$ is stably trivializable. When $n$ is even, we have $\chi(L_k)=2\chi(W_k)=2$. Moreover, we have $\pi_i(L_k)\to \pi_i(W_k)$ is an isomorphism when $i<\frac{n+1}{2}$. In particular, we have $\pi_1(L_k)=0$ and $H^2(L_k)=\Z^{2k}$. When $n>5$ odd, we have $\chi_{\frac{1}{2}}(L_k)=1+2k+2k=1\mod 2$ and when $n=5$, we have $\chi_{\frac{1}{2}}(L_k)=1+2k=1\mod 2$. In other words, all the conditions can be arranged.
\end{enumerate}

\begin{proof}[Proof of Theorem \ref{thm:main}]
Let $\Lambda_k$ denote the Legendrian boundary of $\hat{L}_k$. We claim the contact manifold $Y_k$ obtained as the boundary of the Weinstein domain given by attaching a handle along $\Lambda_k\#\Lambda_k$ is different for different $k$. Then the theorem follows from the same proof of \cite[Theorem 1.1]{Nonsimple}. First note that $\hat{L}_k\sqcup \hat{L}_k$ is a flexible Lagrangian in $D^{2n}\natural D^{2n}$, where $\natural$ is the boundary connected sum. By \cite{ambient}, there is a Lagrangian cobordism from $\Lambda_k\sqcup\Lambda_k$ to $\Lambda_k\#\Lambda_k$ in the symplectization of $(S^{2n-1},\xi_{\std})$, such that the Lagrangian cobordism is also flexible. More precisely, the Legendrian sum, i.e.\ the ambient $0$-surgery in \cite{ambient}, can be understood as a two-steps procedure: we first attach a $1$-handle to $(S^{2n-1},\xi_{\std})$, where the attaching isotropic sphere $S^0$ is two points from each component of $\Lambda_k\sqcup\Lambda_k$. The resulting Weinstein cobordism contains a Lagrangian cobordism from  $\Lambda_k\sqcup\Lambda_k$ to the connected sum. Then we attach another $2$-handle, where the attaching circle is the union of the core of the $1$-handle and the isotropic arc in the construction of the Legendrian sum, without changing the Lagrangian cobordism. This handle will cancel the previously attached $1$-handle symplectically, hence we get a flexible Lagrangian cobordism from $\Lambda_k\sqcup\Lambda_k$ to $\Lambda_k\#\Lambda_k$ in the symplectization of $(S^{2n-1},\xi_{\std})$. As a consequence, $\Lambda_k\#\Lambda_k$ has a flexible Lagrangian filling by $\widehat{L_k\#L_k}$, i.e.\ $L_k\# L_k$ minus a disk. Therefore the resulted contact manifold $Y_k$, which is almost contactomorphic to $(S^*TS^n,\xi_{\std})$, has a Weinstein filling $W_k$ from attaching subcritical/flexible handles to $T^*(L_k\#L_k)$.

\begin{claim}
The flexible cobordism between $ST^*(L_k\#L_k)$ and $Y_k$ has no $1$-handles nor $n$-handles.
\end{claim}
\begin{proof}
Let $C_k$ denote the cobordism. By van Kampen theorem, we have $\pi_1(C_k)=\pi_1(ST^*(L_k\#L_k))=0$. Now since $L_k\#L_k\to W_k$ induces an isomorphism on $n$-th homology, the long exact sequences of $(W_k,T^*(L_k\#L_k))$ and excision implies that $H_n(C_k,ST^*(L_k\#L_k))=H_{n-1}(C_k,ST^*(L_k\#L_k))=0$. By Smale's simplification of the handle representation, we know that $C_k$ has a handle decomposition without $1$-handles nor $n$-handles since $\dim C_k=2n\ge 8$ \cite[\S 7, 8]{MR0190942}. Finally, since $C_k$ is flexible, such topological handle decomposition can be presented in a symplectic way \cite[Chapter 14]{MR3012475}.
\end{proof}

\begin{claim}
$Y_k$ is asymptotically dynamically convex. 
\end{claim}
\begin{proof}
Since $\dim n \ge 4$, $ST^*(L_k\#L_k)$ is tautologically asymptotically dynamically convex. $ST^*(L_k\#L_k)$ is simply connected, the attachment of $2,\ldots, n-1$-subcritical handles does not change the  asymptotically dynamical convexity by \cite[Theorem 3.14]{ADC}. In view of the first claim, we have $Y_k$ is also  asymptotically dynamically convex. 
\end{proof}

\begin{claim}
$SH^{n-1}_+(W_k;\Q)$ are different for different $k$.
\end{claim}
\begin{proof}
    Since $L_k$ is stably trivializable, $L_k$ is spin. Hence $L_k\#L_k$ is also spin, as $w_2(L_k\# L_k)=w_2(L_k)\oplus w_2(L_k) = 0 \in H^2(L_k;\Z/2)\oplus H^2(L_k;\Z/2)=H^2(L_k\#L_k;\Z/2)$. Since $W_k$ is obtained from $T^*(L_k\#L_k)$ by attaching subcritical handles, by \cite{subcritical}, we have  $SH^*(W_k;\Q)=SH^*(T^*(L_k\#L_k);\Q)$. Then by the Viterbo isomorphism  \cite{MR2190223,MR3444367,MR2276534,MR1726235}, we have $SH^*(W_k;\Q)=SH^*(T^*(L_k\#L_k);\Q)=H_{n-*}(\Lambda (L_k\# L_k);\Q)$ and $SH^*_+(W_k;\Q)=H_{n-*}(\Lambda (L_k\# L_k), L_k\#L_k;\Q)$. Since $L_k\#L_k$ is simply connected, Sullivan's minimal model $V_k$ of $L_k\# L_k$ has exactly $4k=\rank H^2(L_k\#L_k;\Q)$ generators $x_1,\ldots,x_{4k}$ in degree $2$, which are also closed. Then by \cite{MR455028}, $H_{1}(\Lambda(L_k\# L_k), L_k\#L_k;\Q)=H^{1}(\Lambda(L_k\# L_k), L_k\#L_k;\Q)=H^1(\bigwedge(V_k\oplus sV_k)/\bigwedge V_k)=\Q^{4k}$, generated by $sx_1,\ldots,sx_{4k}\in sV$, where $sV=V[1]$.
\end{proof}

Now since $SH^*_+(W_k;\Q)$ is a contact invariant for those  asymptotically dynamically convex manifolds with Weinstein fillings \cite[Proposition 3.8]{ADC}, we know that $Y_k$ are different contact manifolds, and the theorem follows.

\end{proof}

\section{Discussion of the $n=3$ case}
The fundamental difficulty to apply the above argument to the $n=3$ case is that we never have asymptotically dynamical convexity unless $L=S^3$.
\begin{proposition}\label{prop:not}
Assume the Liouville domain $W$ has vanishing first Chern class and $\partial W$ is simply connected. If $SH^{m}_{+,S^1}(W)\ne 0$ for some $m\ge 2n-3$, then $\partial W$ is not asymptotically dynamically convex.
\end{proposition}
\begin{proof}
Assume $\partial W$ is asymptotically dynamically convex, i.e.\ there are nesting exact subdomains $\ldots \subset W_n\subset \ldots W_1=W$ with $W_i$ Liouville homotopic to $W$, such that there exist $D_1<\ldots < D_n<\ldots  \to \infty$ with the Reeb orbits on $\partial W_k$ of period up to $D_k$ are non-degenerate and $\mu_{CZ}+n-3>0$. By \cite{MR3734608}, there is a spectral sequence converging to $SH^*_{+,S^1}(W_k)$ with the first page spanned by Reeb orbits of $\partial W_k$ with grading $n-\mu_{CZ}$. The same spectral sequence holds for filtered $S^1$ equivariant positive symplectic cohomology $SH^{*,<D_k}_{+,S^1}(W_k)$ generated by Reeb orbits of period up to $D_k$ and is compatible with continuation maps and Viterbo transfer maps. As a consequence, $SH^{*,<D_k}_{+,S^1}(W_k)$ is supported in grading $*<2n-3$, and the same holds for $SH^*_{+,S^1}(W)=\varinjlim SH^{*,<D_k}_{+,S^1}(W_k)$, which contradicts the condition.
\end{proof}

\begin{proposition}\label{prop:ADC}
Let $W(L)$ denote the exotic $T^*S^3$ obtained from $T^*L$ for any oriented $3$-fold $L$ as in \cite[Theorem 4.7]{flexible}, i.e.\ $L\ne S^3$. Then $\partial W(L)$ is not asymptotically dynamically convex if $L\ne S^3$. 
\end{proposition}
\begin{proof}
Since $L\ne S^3$, $L$ is not simply connected. If the contact form on $ST^*L$ is induced from a Riemannian metric, then there is a non-trivial geodesic loop in each non-trivial conjugacy class of $\pi_1(L)$ minimizing the length (or the energy functional). Under non-degeneracy assumptions, the Morse-index of such geodesic loop, i.e.\ the Conley-Zehnder index of the corresponding Reeb orbit, is $0$, which is a borderline failure for  asymptotically dynamical convexity. Moreover, this loop contributes non-trivially in $H^{S^1}_0(\Lambda L)$, the $S^1$ equivariant homology of the free loop space. However, it is important to note that such Reeb orbit is not contractible, hence is not considered in the definition of asymptotically dynamical convexity for $T^*L$. On the other hand, those homotopy classes will be trivialized after attaching the flexible cobordism to obtain $W(L)$. We claim the unique trivialization $\eta$ of $\det_{\C}W(L)$ restricted to $L$ is the natural trivialization $\det_{\C} T^*L = \det_{\R}L \otimes \C$ used to obtain the $\Z$-graded isomorphism $SH^*_{S^1}(T^*L)=H^{S^1}_{n-*}(\Lambda L)$. This can be seen from the construction of $W(L)$ as follows. It follows from \cite[Corollary 4.5]{flexible} that $\eta$ restricted $\widehat{L}\subset \C^n$ is the restriction of the natural one. Since $H^1(\partial \widehat{L})=0$, there is a unique way glue the trivialization on $\widehat{L}$ and the core of the critical handle. Hence the claim follows. As a consequence, we have $SH^*_{+,S^1}(W(L))=H^{S^1}_{3-*}(\Lambda L, L)$ as $\Z$-graded spaces by \cite{surgery}, where the positive $S^1$-equivariant symplectic cohomology is graded by $3$ minus the Conley-Zehnder index.  In particular, we have $SH^{3}_{+,S^1}(W(L))\ne 0$, hence $\partial W(L)$ is not asymptotically dynamically convex by Proposition \ref{prop:not}.
\end{proof}

\begin{remark}
It is also not clear if we actually need critical handles to build $W(L)$ from $T^*L$, as $ST^*L$ and the flexible cobordism are no longer simply connected. Presumably, there should be a definitive answer to this question if one takes a closer look at the topological type of the flexible cobordism.
\end{remark}

The absence of asymptotically dynamical convexity makes it  hard to tell the contact boundaries apart. However, in some special cases, we can still argue the contact boundaries are different.

\begin{proposition}\label{prop:lens}
Let $W_k$ denote the exotic $T^*S^3$ obtained from $T^*L(k,1)$ for the lens space $L(k,1)$ in \cite[Theorem 4.7]{flexible}. Then $\partial W_k$ is different from $\partial W_{k'}$ if $k\ne k'$. 
\end{proposition}
\begin{proof}
Let $g$ be the round metric on $L(k,1)$, then for each nontrivial homotopy class of loops, the closed geodesics in this homotopy class is parameterized by a family of $S^2$ indexed by $\N$, with only one of them (i.e.\ the simple loops) realizing the local minimum of the energy functional. Then after a small perturbation of the contact form induced from the round metric, there is a unique Reeb orbit $\gamma_i$, for $i\in \Z/k \backslash \{0\}$, with Conley-Zehnder index $0$ in each nontrivial homotopy class of loops, and all others have  Conley-Zehnder indices at least $2$. Moreover, $\gamma_i$ contributes to the nontrivial class in $SH^{3}_{+,S^1}(T^*L(k,1))=H^{S^1}_{0}(\Lambda L(k,1),L(k,1))$ and $\{\gamma_1,\ldots,\gamma_{k-1}\}$ span a basis of the space. Now $W_k$ is obtained from $T^*L(k,1)$ by attaching flexible handles. Following the argument in Proposition \ref{prop:ADC}, even though we have $2$-handles in the flexible cobordism, (the proof of) \cite[Theorem 3.14]{ADC} can still be applied, as the trivialization of complex determinant bundle can be extended  naturally. As a consequence, after attaching all the subcritical handles, we have a contact form on the boundary such that the all Reeb orbits have Conley-Zehnder indices at least $2$ except for $\gamma_i$\footnote{Strictly speaking, we need to work with a nesting family of Liouville domains with an increasing sequence of period threshold as in \cite{ADC}, we omit this for simplicity. Also by a bit abuse of language,  $\gamma_i$ here are the corresponding orbits after the surgery, as we can assume they are away from the surgery region.}. Finally, we attach the critical flexible handles, by \cite[Theorem 3.15]{ADC}, the new contact boundary has new Reeb orbits with Conley-Zehnder index at least $1$. However by the proof of \cite[Theorem 3.15]{ADC}, those orbits with Conley-Zehnder index $1$ are from single Reeb chords with arbitrarily small period from the small zig-zags on loose Legendrians. In particular, we can assume the corresponding Reeb orbits have period smaller that of $\gamma_i$. As a consequence, $\gamma_i$ still contributes non-trivially to $SH^{3}_{+,S^1}(W)$ and spans the space ($\simeq {\Q}^{k-1}$ if we use $\Q$ as the coefficient), for any Weinstein filling $W$ of $\partial W_k$, as they can not be eliminated by those orbits with Conley-Zehnder index $1$. As a consequence, $\partial W_k\ne \partial W_{k'}$ for $k\ne k'$
\end{proof}

\begin{corollary}
For any $n\ge 3$, there exist infinitely many exotic $T^*S^n$ with the standard almost Weinstein structure, such that the contact boundaries are also pairwisely different.
\end{corollary}
\begin{remark}
For $n\ge 9$, Zhao \cite{zhao} proved that for any $N$, there exist $N$ different $2n$-dimensional Weinstein domains that have the same almost Weinstein structure, and the contact boundaries are pairwisely different.
\end{remark}

Another difficulty is that $\pi_1(L\#L))$ has infinitely many conjuacy classes as long as $L\ne S^3$, hence we can not apply the argument in Proposition \ref{prop:lens} to $L\#L$\footnote{Moreover, there may not be any nice Reeb flow on $ST^*(L\#L)$. In view of Meyer's theorem \cite[Theorem 19.4]{MR0163331}, a nice geodesic flow requires, for example, positive Ricci curvature, which in dimension 3, only happens on quotients of $S^3$.}. Moreover, we do not know that the exotic $T^*S^3$ given by $W(L\#L)$ are different from each other, let alone the contact boundary. The former question, in principle, can be answered by understanding $H_*(\Lambda (L\#L))$ along with the rich structures on it. 

\begin{proof}[Proof of Proposition \ref{prop:three}]
Let $L\ne S^3$, it suffices to prove that $\partial W(L\# L)$ is different from $ST^*S^3$ and $\partial (\mathrm{Flex}(T^*S^3))$. This is clear from Proposition \ref{prop:ADC}, as the latter two are both asymptotically dynamically convex.
\end{proof}

\bibliographystyle{plain} 
\bibliography{ref}

\begin{thebibliography}{10}

\bibitem{MR2190223}
Alberto Abbondandolo and Matthias Schwarz.
\newblock On the {F}loer homology of cotangent bundles.
\newblock {\em Comm. Pure Appl. Math.}, 59(2):254--316, 2006.

\bibitem{MR3444367}
Mohammed Abouzaid.
\newblock Symplectic cohomology and {V}iterbo's theorem.
\newblock In {\em Free loop spaces in geometry and topology}, volume~24 of {\em
  IRMA Lect. Math. Theor. Phys.}, pages 271--485. Eur. Math. Soc., Z\"{u}rich,
  2015.

\bibitem{surgery}
Fr\'{e}d\'{e}ric Bourgeois, Tobias Ekholm, and Yasha Eliashberg.
\newblock Effect of {L}egendrian surgery.
\newblock {\em Geom. Topol.}, 16(1):301--389, 2012.
\newblock With an appendix by Sheel Ganatra and Maksim Maydanskiy.

\bibitem{Nonsimple}
Roger Casals and John~B. Etnyre.
\newblock Non-simplicity of isocontact embeddings in all higher dimensions.
\newblock {\em Geom. Funct. Anal.}, 30(1):1--33, 2020.

\bibitem{subcritical}
Kai Cieliebak.
\newblock Handle attaching in symplectic homology and the chord conjecture.
\newblock {\em J. Eur. Math. Soc. (JEMS)}, 4(2):115--142, 2002.

\bibitem{MR3012475}
Kai Cieliebak and Yakov Eliashberg.
\newblock {\em From {S}tein to {W}einstein and back}, volume~59 of {\em
  American Mathematical Society Colloquium Publications}.
\newblock American Mathematical Society, Providence, RI, 2012.
\newblock Symplectic geometry of affine complex manifolds.

\bibitem{cote2020homological}
Laurent C\^{o}t\'{e} and Francois-Simon Fauteux-Chapleau.
\newblock Homological invariants of codimension 2 contact submanifolds.
\newblock {\em arXiv preprint arXiv:2009.06738}, 2020.

\bibitem{ambient}
Georgios Dimitroglou~Rizell.
\newblock Legendrian ambient surgery and {L}egendrian contact homology.
\newblock {\em J. Symplectic Geom.}, 14(3):811--901, 2016.

\bibitem{h-principle}
Y.~Eliashberg and N.~Mishachev.
\newblock {\em Introduction to the {$h$}-principle}, volume~48 of {\em Graduate
  Studies in Mathematics}.
\newblock American Mathematical Society, Providence, RI, 2002.

\bibitem{flexible}
Yakov Eliashberg, Sheel Ganatra, and Oleg Lazarev.
\newblock Flexible {L}agrangians.
\newblock {\em Int. Math. Res. Not. IMRN}, (8):2408--2435, 2020.

\bibitem{contact}
Hansj\"{o}rg Geiges.
\newblock {\em An introduction to contact topology}, volume 109 of {\em
  Cambridge Studies in Advanced Mathematics}.
\newblock Cambridge University Press, Cambridge, 2008.

\bibitem{MR3734608}
Jean Gutt.
\newblock The positive equivariant symplectic homology as an invariant for some
  contact manifolds.
\newblock {\em J. Symplectic Geom.}, 15(4):1019--1069, 2017.

\bibitem{ADC}
Oleg Lazarev.
\newblock Contact manifolds with flexible fillings.
\newblock {\em Geom. Funct. Anal.}, 30(1):188--254, 2020.

\bibitem{MR0163331}
J.~Milnor.
\newblock {\em Morse theory}.
\newblock Annals of Mathematics Studies, No. 51. Princeton University Press,
  Princeton, N.J., 1963.
\newblock Based on lecture notes by M. Spivak and R. Wells.

\bibitem{MR0190942}
John Milnor.
\newblock {\em Lectures on the {$h$}-cobordism theorem}.
\newblock Princeton University Press, Princeton, N.J., 1965.
\newblock Notes by L. Siebenmann and J. Sondow.

\bibitem{MR2276534}
D.~A. Salamon and J.~Weber.
\newblock Floer homology and the heat flow.
\newblock {\em Geom. Funct. Anal.}, 16(5):1050--1138, 2006.

\bibitem{MR455028}
Micheline Vigu\'{e}-Poirrier and Dennis Sullivan.
\newblock The homology theory of the closed geodesic problem.
\newblock {\em J. Differential Geometry}, 11(4):633--644, 1976.

\bibitem{MR1726235}
C.~Viterbo.
\newblock Functors and computations in {F}loer homology with applications. {I}.
\newblock {\em Geom. Funct. Anal.}, 9(5):985--1033, 1999.

\bibitem{zhao}
Mu~Zhao.
\newblock Stein domains with exotic contact boundaries.
\newblock {\em arXiv preprint arXiv:2002.05524}, 2020.

\bibitem{vanishing}
Zhengyi Zhou.
\newblock Vanishing of symplectic homology and obstruction to flexible
  fillability.
\newblock {\em Int. Math. Res. Not. IMRN}, (23):9717--9729, 2020.

\end{thebibliography}
\Addresses
\end{document}